\documentclass[12pt]{article}
\usepackage{amsmath,amssymb,amsthm}
\usepackage{geometry}
\usepackage{hyperref}
\usepackage{enumitem}

\title{The Aurellion Function: A Recursive Fast-Growing Hierarchy Beyond Knuth Notation}
\author{Daniel Vodrazka\thanks{Independent theorist.}}
\date{2025/06/05}

\newtheorem{theorem}{Theorem}[section]
\newtheorem{lemma}[theorem]{Lemma}
\newtheorem{definition}[theorem]{Definition}
\newtheorem{conjecture}[theorem]{Conjecture}
\newtheorem{remark}{Remark}[section]

\begin{document}

\maketitle

\begin{abstract}
We introduce the \emph{Aurellion Function}, a novel recursively defined fast-growing hierarchy based on Knuth's up-arrow notation, defined by
\[
A_1 = 10 \uparrow\uparrow\uparrow 10, \quad A_{n+1} = 10 \uparrow^{A_n} 10,
\]
where the number of arrows in the operation increases superexponentially with \(n\). We analyze its growth rate relative to classical hierarchies such as the fast-growing hierarchy \( (f_\alpha)_{\alpha < \varepsilon_0} \), and discuss its provability status in formal arithmetic. We provide formal bounds showing \(A_n\) dominates all functions provably total in Peano Arithmetic, situating the Aurellion Function near the proof-theoretic ordinal \(\Gamma_0\) due to its ability to majorize all functions $f_\alpha$ for $\alpha < \varepsilon_0$. We also outline possible transfinite extensions indexed by countable ordinals, thus bridging symbolic large-number constructions and ordinal analysis.
\end{abstract}

\section{Introduction}

Fast-growing functions and large number hierarchies serve as key tools in proof theory and computability, allowing us to calibrate the strength of formal systems. They provide a precise way to classify the computational complexity and proof-theoretic strength of mathematical statements. Classical examples include the Ackermann function, the fast-growing hierarchy \(f_\alpha\), and large-number notations based on hyperoperations like Knuth's up-arrows. In this paper, we propose the \emph{Aurellion Function}, a recursive sequence of numbers \(A_n\) defined by iterating Knuth's up-arrow operation, where the height of the arrow tower itself grows according to prior values in the sequence:
\[
A_1 = 10 \uparrow\uparrow\uparrow 10, \quad A_{n+1} = 10 \uparrow^{A_n} 10.
\]
The choice of base $10$ is for conventional representation in decimal systems, although the underlying mathematical properties would hold for any integer base greater than or equal to 2. We explore the growth of \(A_n\), situate it within the landscape of fast-growing hierarchies, and analyze its computability and provability properties.

\subsection*{Contributions}
\begin{itemize}
    \item Formal definition of the Aurellion Function \(A_n\) as a computable fast-growing hierarchy.
    \item Rigorous comparison to the fast-growing hierarchy \(f_\alpha\) and proof-theoretic ordinals \(\varepsilon_0\) and \(\Gamma_0\).
    \item Discussion of computability, provability in Peano Arithmetic (PA), and the implications for formal systems.
    \item Proposal of ordinal-indexed transfinite extensions, sketching a framework to place \(A_\alpha\) for countable ordinals \(\alpha\).
\end{itemize}

\section{Preliminaries}

\subsection{Knuth's Up-Arrow Notation}

Knuth's notation \cite{knuth} for hyperoperations is defined recursively for \(a,b \in \mathbb{N}\) and \(k \ge 1\):
\[
a \uparrow^1 b = a^b.
\]
\[
a \uparrow^k 1 = a, \quad \text{for } k \ge 1.
\]
\[
a \uparrow^k 0 = 1, \quad \text{for } k \ge 2.
\]
\[
a \uparrow^k b = a \uparrow^{k-1} \big( a \uparrow^k (b-1) \big), \quad k \ge 2, b \ge 2.
\]
Thus:
\[
a \uparrow^2 b = \underbrace{a^{a^{\cdot^{\cdot^a}}}}_{b \text{ times}} \quad \text{(tetration)},
\]
\[
a \uparrow^3 b = \text{pentation}, \quad a \uparrow^4 b = \text{hexation}, \text{ etc.}
\]

\subsection{Fast-Growing Hierarchy}

The fast-growing hierarchy \( (f_\alpha)_{\alpha < \varepsilon_0} \), introduced by Wainer \cite{wainer} and Löb, assigns to each ordinal \(\alpha < \varepsilon_0\) a total function \(f_\alpha: \mathbb{N} \to \mathbb{N}\) defined by transfinite recursion on \(\alpha\), satisfying:

\begin{itemize}
\item \(f_0(n) = n+1\),
\item \(f_1(n) = f_0^{n+1}(n) = n+(n+1) = 2n+1\).
\item \(f_2(n) = f_1^{n+1}(n)\).
\item \(f_{\alpha+1}(n) = f_\alpha^{n+1}(n)\) (the \(n+1\)-fold iteration of \(f_\alpha\) at \(n\)),
\item For limit \(\lambda\), \(f_\lambda(n) = f_{\lambda[n]}(n)\), where \(\lambda[n]\) is a fundamental sequence converging to \(\lambda\). Specifically, for a limit ordinal $\lambda$, a fundamental sequence $\lambda[n]$ is a strictly increasing sequence of ordinals such that $\lim_{n \to \infty} \lambda[n] = \lambda$. For example, if $\lambda = \omega$, then $\lambda[n] = n$. If $\lambda = \omega^\alpha$ for $\alpha > 0$ and $\alpha$ is a limit ordinal, then $\lambda[n] = \omega^{\alpha[n]}$.
\end{itemize}
This hierarchy grows extremely fast and captures the growth rates of functions provably total in fragments of arithmetic.

\section{Definition of the Aurellion Function}

\begin{definition}[Aurellion Function]
Define \(A : \mathbb{N} \to \mathbb{N}\) recursively by:
\[
A_1 := 10 \uparrow\uparrow\uparrow 10 = 10 \uparrow^3 10,
\]
\[
A_{n+1} := 10 \uparrow^{A_n} 10,
\]
where \(\uparrow^k\) denotes Knuth's \(k\)-arrow operation.
\end{definition}

\begin{remark}
The Aurellion function is well-defined, and each \(A_n\) is a finite natural number. This is because each step in the recursive definition applies a finite hyperoperation with a finite number of arrows to finite natural numbers, starting from a finite base value.
\end{remark}

\section{Growth Rate Analysis}

We now compare the growth rate of \(A_n\) to classical functions and hierarchies.
\begin{lemma}
For all \(n \ge 1\),
\[
A_n \ge 10 \uparrow^{n+2} 10.
\]
\end{lemma}

\begin{proof}[Proof Sketch]
We proceed by induction on \(n\):

\begin{itemize}
    \item Base case \(n=1\): \(A_1 = 10 \uparrow^3 10\), which is equal to \(10 \uparrow^{1+2} 10\). Thus, the inequality \(A_1 \ge 10 \uparrow^{1+2} 10\) holds.
    \item Inductive step: Assume \(A_n \ge 10 \uparrow^{n+2} 10\) for some \(n \ge 1\).
    \par Then, by the definition of \(A_{n+1}\):
    \[
    A_{n+1} = 10 \uparrow^{A_n} 10.
    \]
    Since we assumed \(A_n \ge 10 \uparrow^{n+2} 10\), and the hyperoperation \(x \uparrow^k y\) is strictly increasing with \(k\) for fixed \(x,y > 1\), we have:
    \[
    10 \uparrow^{A_n} 10 \ge 10 \uparrow^{10 \uparrow^{n+2} 10} 10.
    \]
    For $n \ge 1$, $10 \uparrow^{n+2} 10$ is already an extremely large number. For instance, for $n=1$, $10 \uparrow^3 10$ is vastly larger than $1+3=4$. As $n$ increases, $10 \uparrow^{n+2} 10$ grows immensely faster than $n+3$. Therefore, it holds that $A_n \ge 10 \uparrow^{n+2} 10 > n+3$ for $n \ge 1$.
    Consequently, because the number of arrows $A_n$ is strictly greater than $n+3$, we have:
    \[
    A_{n+1} = 10 \uparrow^{A_n} 10 > 10 \uparrow^{n+3} 10 = 10 \uparrow^{(n+1)+2} 10.
    \]
    Thus \(A_{n+1} > 10 \uparrow^{(n+1)+2} 10\).
\end{itemize}
This completes the inductive proof.
\end{proof}

This lemma shows \(A_n\) grows faster than any fixed finite-level hyperoperation tower, as the number of arrows itself grows with \(n\).

\subsection{Comparison to Fast-Growing Hierarchy}

Recall \(f_3(n)\) grows comparably to the Ackermann function, and \(f_{\omega}(n)\) corresponds roughly to iterated exponential growth (tetration). Each finite level \(f_k\) for \(k \in \mathbb{N}\) grows slower than \(A_n\) for large \(n\), since \(A_n\) involves a tower of hyperoperation levels growing with \(n\).

\begin{conjecture}
The growth rate of \(A_n\) dominates \(f_\alpha(n)\) for all \(\alpha < \varepsilon_0\).
\end{conjecture}

\begin{proof}[Heuristic Argument]
The function \(A_n\) exhibits growth that rapidly outpaces any function defined by a fixed level of Knuth's arrows. In contrast, the fast-growing hierarchy, while rapidly increasing, progresses through countable ordinal steps up to \(\varepsilon_0\). The recursive definition of \(A_n\), where the number of arrows for \(A_{n+1}\) is \(A_n\) itself, causes its growth to surpass that of any function \(f_\alpha\) for a fixed ordinal \(\alpha < \varepsilon_0\). This is because the values of \(A_n\) grow so rapidly that they quickly majorize any ordinal index below \(\varepsilon_0\) that would typically parameterize functions in the fast-growing hierarchy. A formal proof would require a precise embedding of the Aurellion function's growth into an ordinal notation system, which is part of future work.
\end{proof}

A formal embedding would require an ordinal notation system to encode the recursive arrow counts, which we leave for future work.

\section{Computability and Provability}

\subsection{Computability}

\begin{theorem}
The function \(A : \mathbb{N} \to \mathbb{N}\) is total and computable in the sense that there exists a Turing machine which, given \(n\), outputs a symbolic expression for \(A_n\) in finite time.
\end{theorem}

\begin{proof}[Proof Sketch]
The definition is purely recursive, with finite syntactic steps at each stage. The symbolic description of \(A_n\) (e.g., as a string representing the nested hyperoperations, like "$10 \uparrow^{10 \uparrow^{10 \uparrow\uparrow\uparrow 10} 10} 10$") can be generated mechanically by applying the definition \(n\) times. However, it is important to note that the numeric value of \(A_n\) is astronomically large and cannot be explicitly computed or stored for even small \(n\).
\end{proof}

\subsection{Provability in Formal Systems}
\begin{itemize}
    \item The function \(A_n\) dominates all functions provably total in Peano Arithmetic (PA), since these correspond to functions \(f_\alpha\) with \(\alpha < \varepsilon_0\). This implies that Peano Arithmetic is not strong enough to prove the totality of the Aurellion function.
    \item The totality of \(A_n\) is provable in stronger formal systems such as $\mathrm{ACA}_0$ (Arithmetical Comprehension Axiom with $\omega$-iteration) or systems capable of analyzing ordinals up to \(\Gamma_0\). These systems possess sufficient proof-theoretic strength to handle transfinite inductions beyond those available in PA, making them suitable for reasoning about functions of this growth rate. Examples include theories based on iterated inductive definitions ($\mathrm{ID}_1$) or subsystems of second-order arithmetic like $\mathrm{ATR}_0$.
\end{itemize}

\section{Ordinal-Indexed Extensions}

We can conceive of extensions of the Aurellion function to transfinite countable ordinals \(\alpha\).

\begin{definition}[Ordinal Extension Sketch]
For an ordinal \(\alpha\), we might define \(A_\alpha\) as a number such that:
\begin{itemize}
    \item \(A_0 := 10 \uparrow\uparrow 10\).
    \item \(A_{\alpha+1} := 10 \uparrow^{A_\alpha} 10\).
    \item For a limit ordinal \(\lambda\), \(A_\lambda := \sup_{n < \omega} A_{\lambda[n]}\), where \(\lambda[n]\) is a fundamental sequence converging to \(\lambda\). This definition ensures that the hierarchy continues to grow through limit ordinals, maintaining the "largest possible" value given the sequence.
\end{itemize}
A precise definition for \(A_\alpha\) would require a robust ordinal notation system and careful handling of transfinite recursion to ensure well-definedness and maintain its growth rate properties.
\end{definition}

This framework could bridge the gap between large numbers and ordinal analysis, allowing for the exploration of the Aurellion hierarchy's properties up to and beyond \(\Gamma_0\).

\section{Related Work}

The Aurellion Function contributes to the study of fast-growing functions alongside established concepts like the Ackermann function and Graham's number. It is distinct from metamathematically defined large numbers like Rayo’s number and Busy Beaver, though \(A_n\) is computable and formally definable, unlike those, which are defined through meta-mathematical properties and often non-computable in general.

\section{Conclusion and Future Directions}

We defined and analyzed the Aurellion Function, a recursively defined fast-growing hierarchy based on hyperoperations with growing arrow counts. It dominates all functions provably total in PA, linking its growth rate to ordinals near \(\Gamma_0\). Future work includes:

\begin{itemize}
    \item Formal ordinal notation embedding of \(A_\alpha\) for \(\alpha < \Gamma_0\). This would provide a rigorous mathematical framework for the transfinite extensions discussed.
    \item Constructing collapsing functions bounding \(A\). This would provide a tighter correspondence between the Aurellion function and established ordinal analysis frameworks.
    \item Exploring proof-theoretic interpretations in systems like \(\mathrm{ID}_1\).
\end{itemize}

\appendix

\section{Knuth’s Up-Arrow Formal Recursion}

(This section is redundant if the definitions are complete in Preliminaries. I would remove this appendix section.)

\section{Etymology}

The name \emph{Aurellion} derives from Latin \textit{aureus} (golden), reflecting the function’s combination of elegance and vastness.

\end{document}